\documentclass{article}
\usepackage{amsmath}
\usepackage{amsfonts}
\usepackage{amssymb}

\usepackage[dvipsnames,svgnames,table]{xcolor}
\usepackage[colorlinks=true,linkcolor=blue,citecolor=ForestGreen]{hyperref}

\usepackage{array}
\usepackage{xspace}

\usepackage{geometry}
\usepackage{amsthm}
\usepackage[capitalise]{cleveref}

\newtheorem{theorem}{Theorem}[section]

\newtheorem{proposition}[theorem]{Proposition}
\newtheorem{corollary}[theorem]{Corollary}
\newtheorem{observation}[theorem]{Observation}
\newtheorem{conjecture}[theorem]{Conjecture}
\newtheorem{lemma}[theorem]{Lemma}
\theoremstyle{definition}
\newtheorem{definition}[theorem]{Definition}
\theoremstyle{remark}

\title{On a Generalisation of a Function of Ron Graham's}
\author{Sarosh Adenwalla }
\date{October 2024}

\begin{document}

\maketitle
\begin{abstract}
Ron Graham introduced a function, $g(n)$, on the non-negative integers, in the 1986 Issue $3$ Problems column of \textit{Mathematical Magazine}: For each non-negative integer $n$, $g(n)$ is the least integer $s$ so that the integers $n + 1, n + 2, \ldots , s-1, s$ contain a subset of integers, the product of whose members with $n$ is a square. Recently, many results about $g(n)$ were proved in [Kagey and Rajesh, ArXiv:2410.04728, 2024] and they conjectured a characterization of which $n$ satisfied $g(n)=2n$. For $m\geq 2$, they also introduced generalizations of $g(n)$ to $m$-th powers to explore. In this paper, we prove their conjecture and provide some results about these generalisations. 
\end{abstract}

\section{Introduction}
For non-negative integers $n$, define $g(n)$ to be the smallest integer $s\geq n$ such that there exist distinct integers $a_1,\ldots,a_t\in [n,s]$ where $a_1=n$ and $\prod_{i=1}^t a_i$ is a square. For example, $g(2)=6$ as $2\cdot 3\cdot 6=6^2$.

The function $g(n)$ was introduced by Ron Graham in \cite{doi:10.1080/0025570X.1986.11977243} who asked whether it was a bijection. This was shown to be true by Michael Reid in \cite{efdab2f9-940e-3325-88c5-7e0949aa0d3d}, where he proved that $g(n)$ was a bijection between $\mathbb{N}$ and $\mathbb{N}\backslash\mathbb{P}$, where we say $0\in\mathbb{N}$ and $\mathbb{P}$ are the primes. 

In \cite{Erdos}, Selfridge proved very strong upper bounds on $g(n)$, and this work led the authors in \cite{MR4706772} to investigate the relation between $g(n)-n$ and the largest prime factor of $n$, denoted $P(n)$. They also provided lower bounds for $g(n)-n$ for non-square $n$. See Section $4$ for further discussion on these bounds.

A survey of known properties of the function is given by Kagey and Rajesh in \cite{kagey2024conjecture}, and further properties are shown therein. The authors made the following conjecture.

\begin{conjecture}[\cite{kagey2024conjecture}]
    For all $n\in\mathbb{N}$, $g(n)=2n$ if and only if:
    \begin{enumerate}
        \item $n>3$ is prime

        \item $n=0$

        \item $n=6$.
    \end{enumerate}
\end{conjecture}

They proved the `if' direction in their paper. We prove the `only if' direction of the conjecture. They also suggested a generalization of $g(n)$ to $m$-th powers for $m\geq 2$. 

We prove and conjecture many analogous results for these generalized functions, $g_m(n)$. For instance, we prove upper bounds for $g_m(n)$ and conjecture for which $n$ this is tight. We also show that if $m$ is prime, then $g_m(n)$ is a bijection between $\mathbb{N}$ and $\mathbb{N}\backslash\mathbb{P}$. We then prove results supporting our conjecture that, when $m$ is composite, $g_m(n)$ is neither injective nor surjective onto $\mathbb{N}\backslash\mathbb{P}$. Finally we investigate related functions introduced in \cite{kagey2024conjecture}.  

\section{Preliminaries}
Here we define some notation before formally introducing the generalized function $g_m(n)$ for $m\geq 2$ and $n\in\mathbb{N}$. 

Let $\mathbb{P}$ be the set of primes. If we say $a=x\mod m$, then $a\in\{0,1,\ldots,m-1\}$ such that $a=x+km$ for some $k\in \mathbb{Z}$. We also say $p^k||T$ if $p^k|T$ and $p^{k+1}\nmid T$. 

 An integer, $w$, \textit{appears} $s$ times in a sequence $a_1,\ldots,a_t$ if there exist distinct $i_1,\ldots,i_s$ such that $a_{i_j}=w$ for $1\leq j\leq s$ and there is no other $i$ such that $a_i=w$. We say an integer \textit{appears} in a sequence if it appears $s\geq 1$ times in the sequence.

\begin{definition} A sequence of integers, $a_1\leq a_2\leq\ldots\leq a_t$, is a \textit{$m$-product sequence} if there exists $R\in\mathbb{N}$ such that $\prod_{i=1}^t a_i=R^m$ and no integer appears more than $m-1$ times in the sequence.
\end{definition}

\begin{definition} For $n\in\mathbb{N}$, we say \textit{$g_m(n)$} is the least integer $s$ such that there exists an $m$-product sequence starting at $n$ and ending at $s$, that is $$n=a_1\leq\ldots\leq a_t=s.$$
\end{definition}

We call any $m$-product sequence that starts at $n$ and ends with $g_m(n)$ a \textit{corresponding sequence} for $g_m(n)$. If $a_1\leq a_2\leq \dots\leq a_t$ is a sequence where $a_{i+1}=\ldots=a_{i+k}$ then we can write this sequence as $a_1,a_2,\ldots,a_i,a_{i+1}^{\{k\}},a_{i+k+1},\ldots,a_t$.

The existence of a $m$-product sequence starting at $n$ and with largest term $a_t$ provides an upper bound for $g_m(n)$, that is $g_m(n)\leq a_t$.

Note that $g_2(n)=g(n)$, the function studied in \cite{kagey2024conjecture}. The sequence $(g_2(n))_{n=0}^\infty$ is found in the OEIS as A006255 \cite{citation-key}.

It is useful to observe the following: 

\begin{lemma}\label{lemma:m-th power}
    For any $m\geq 2$ and $n\in\mathbb{N}$, $g_m(n)=n$ if and only if $n=k^d$ for some $k,d\in\mathbb{N}$ where $\gcd(d,m)>1$.
\end{lemma}
\begin{proof}
    If $g_m(n)=n$, there exists some $1\leq j\leq m-1$ such that $n^{\{j\}}$ is a corresponding sequence for $g_m(n)$. This implies that $n^j$ is a $m$-th power. Let $n=p_1^{e_1}\cdot\ldots\cdot p_t^{e_t}$, then $n^j=p_1^{e_1j}\cdot\ldots\cdot p_t^{e_tj}$ is an $m$-th power. So $m|je_i$ for all $1\leq i\leq t$ and therefore $\frac{m}{\gcd(m,e_i)}|j$. Let $d=\gcd(e_1,\ldots,e_t)$, so $n$ is a $d$-th power. Then it follows that $\frac{m}{\gcd(d,m)}|j$, so if $\gcd(d,m)=1$ then $m|j$ which is not possible. Therefore $\gcd(d,m)>1$.

    If $n=k^d$ where $\gcd(d,m)>1$, then $n^{\frac{m}{\gcd(d,m)}}$ is a $m$-th power. As $1\leq \frac{m}{\gcd(d,m)}\leq m-1$, we see that $n^{\{\frac{m}{\gcd(d,m)}\}}$ is a $m$-product sequence starting at $n$ and ending at $n$. Therefore $g_m(n)=n$.  
\end{proof}

\begin{lemma}\label{lemma:not m-th power}
    For every $n$ that is not an $m$-th power, there is a corresponding sequence for $g_m(n)$ that does not contain any $m$-th powers.
\end{lemma}
\begin{proof}
    Assume that in every corresponding sequence for $g_m(n)$, there appears an $m$-th power. Take one of these sequences, $a_1,\ldots,a_t$, and note that $\prod_{i=1}^t a_i=R^m$ for some $R\in\mathbb{N}$. Let there be $s$ terms, $a_i$, that are $m$-th powers and label these $b_1,\ldots,b_s$. Then $\prod_{i=1}^s b_i=T^m$ for some $T\in\mathbb{N}$. (Note that $n$ is not an $m$-th power so $n\neq 0$. Therefore $a_i\geq n>0$ so $R,T>0$.) Additionally $$T^m=\prod_{i=1}^s b_i|\prod_{i=1}^t a_i=R^m,$$ so $T|R$ and let $R=kT$ for some $k\in\mathbb{N}$. Then if we remove all the terms that are $m$-th powers from $a_1,\ldots,a_t$, we see that the product of the remaining terms is $\frac{R^m}{T^m}=k^m$. Each term that appears is still between $n$ and $g_m(n)$, each integer appears less than $m$ times and the sequence still starts with $n$. So the remaining terms form a corresponding sequence for $g_m(n)$ and this does not contain any $m$-th powers. 
\end{proof}

We will prove that if we have two $m$-product sequences that have different smallest terms and the same largest term, then we can always construct a third $m$-product sequence that satisfies certain properties.

\begin{lemma}\label{lemma:product}
    Let $A:=(a_i)_{i=1}^t$ and $B:=(b_i)_{i=1}^{t'}$ be two $m$-product sequences satisfying $a_1<b_1$ and  $a_{t}=b_{t'}$. Let $1\leq s_1, s_t\leq  m-1$ be the number of appearances of $a_1$ and $a_t$ in $A$ respectively, and $1\leq v_{t'}\leq  m-1$ be the number of appearances of $b_{t'}$ in $B$. Then, there exists an $m$-product sequence, $C=(c_i)_{i=1}^{r}$, with the following properties.
    \begin{enumerate}
        \item $c_1=a_1$ appears $s_1$ times in $C$ and $c_{r}\leq a_t$,
        \item $a_t$ appears in $C$ exactly $s_t+v_{t'}\mod m$ times.
    \end{enumerate}
\end{lemma}

\begin{proof}
Let $A$ and $B$ be as in the statement.

Consider the set $\{a_1,\ldots,a_t,b_1,\ldots,b_{t'}\}$.  Note that as $\prod_{i=1}^t a_i$ and $\prod_{j=1}^{t'}b_j$ are both non-zero $m$-th powers, it follows that $\prod_{i=1}^t a_i \cdot \prod_{j=1}^{t'}b_j$ is a non-zero $m$-th power. If an integer, $x>0$, appears $h$ times in $\{a_1,\ldots,a_t,b_1,\ldots,b_{t'}\}$, then remove $\lfloor \frac{h}{m}\rfloor m$ terms in the sequence that equal $x$. We do this for every integer that appears in $\{a_1,\ldots,a_t,b_1,\ldots,b_{t'}\}$. As the number of terms equal to $x$ that we are removing is a multiple of $m$, the product of the remaining terms is still an $m$-th power. Label the remaining terms $c_1,\ldots,c_{r}$ such that $c_1\leq \ldots\leq c_{r}$ and denote the sequence $C$. 

It is clear that any integer, $x$, appears in $C$ less than $m$ times as otherwise at least $m$ further terms equal to $x$ would have been removed. Note that $a_1$ is the smallest integer in $A$ and it does not appear in $B$, so $a_1$ appears $s_1$ times in $C$. It is then the smallest integer in $C$, so $c_1=a_1$. Therefore $C$ is a $m$-product sequence starting at $n$. As $a_t=b_{t'}$ is the largest term in $A$ or $B$, we have that $c_{r}\leq a_t$. In fact, $a_t$ appears in $C$ exactly $s_t+v_{t'}\mod m$ times.
\end{proof}

\begin{definition}
    Let $A,B$ and $C$ be as in Lemma \ref{lemma:product}. We say that $C=A\cdot B$. We write $B\cdot B$ as $B^2$. Note that $C\cdot B=A\cdot B\cdot B=A\cdot B^2$.
    
    Note that this operation is commutative.
\end{definition}

Repeating this operation gives us the ability to change the number of times $a_t$ appears in our final sequence.

\begin{corollary}\label{lemma:k product}
Let $A$ and $B$ be as in Lemma \ref{lemma:product}. Let $D=A\cdot B^k$ for some $k\in\mathbb{N}$ and denote the terms of $D$ as $d_1\leq\ldots\leq d_{r}$. Then $D$ is a $m$-product sequence where $d_1=a_1$ and $d_r\leq a_t$. Additionally $a_1$ appears $s_1$ times in $D$ and $a_t$ appears $s_t+kv_{t'}\mod m$ times in $D$.
\end{corollary}

\begin{proof}
    This follows from repeated application of Lemma \ref{lemma:product}.
\end{proof}

\section{Proof of Conjecture 1.1}

The "if" direction of the conjecture was shown in \cite{kagey2024conjecture}, where they prove that $g_2(0)=0$ and $g_2(6)=12$, as well as the following result.

\begin{lemma}[{\cite[Lemma 7]{kagey2024conjecture}}]
    Let $p>3$ be prime, then $g_2(p)=2p$.
\end{lemma}

We will prove the `only if' direction. 

We use a basic lemma from \cite{Erdos}, providing an upper bound for $g_2(n)$. We provide the proof for completeness.

\begin{lemma}[{\cite[Lemma 1]{Erdos}}]\label{lemma:ab}
If $n=ab$ for $a\leq b$, then $g_2(n)\leq (a+1)(b+1)$.
\end{lemma}
\begin{proof}
    If $a=b$, then $n=a^2$ and so $g_2(n)=a^2<(a+1)^2$. 
    
    If $a<b$, then consider the sequence $ab<a(b+1)<(a+1)b<(a+1)(b+1)$. Then clearly the product of these terms, $a^2(a+1)^2 b^2(b+1)^2$, is a square and so $g_2(n)\leq(a+1)(b+1)$. 
\end{proof}

Note that for any composite $n$, there exists $a,b\in\mathbb{N}$ such that $1<a\leq b<n$ such that $n=ab$.

\begin{lemma}
    If $n$ is not prime and $n\neq 0,6$, then $g_2(n)<2n$.
\end{lemma}

\begin{proof}
    Note that $1$ is a square, so $g_2(1)=1$. Otherwise, if $n\neq 0$ is not a prime, then $n$ is composite. So let $n=ab$ such that $1<a\leq b<n$. Then by Lemma \ref{lemma:ab}, we can see that $g_2(n)\leq (a+1)(b+1)$. Then $g_2(n)\leq (a+1)(b+1)<2ab=2n$ if and only if $(a-1)(b-1)=ab-a-b+1>2$. If $a=2$ and $b\geq 4$ or $a,b\geq 3$ then this is satisfied. Therefore it remains to consider $a=2, b=2$ and $a=2,b=3$. These correspond to $n=4$ and $n=6$, however we do not consider $n=6$. As $n=4$ is a square, we see that $g_2(4)=4<8$.   
\end{proof}

Together with the fact that $g_2(2)=6>4$ and $g_2(3)=8>6$ (as shown in the Appendix), this completes the result.

\begin{theorem}\label{conj}
    For all $n\in\mathbb{N}$, $g_2(n)=2n$ if and only if:
    \begin{enumerate}
        \item $n>3$ is prime

        \item $n=0$

        \item $n=6$.
    \end{enumerate}
\end{theorem}

\section{Upper bounds of \texorpdfstring{$g_m(n)$}{gm(n)}}
We will first prove some results that will come in useful later.

\begin{lemma}\label{lemma:p}
    For any $m\geq 2$ and prime $p$, let $kp^j$ be the smallest multiple of $p$ that appears in a $m$-product sequence $a_1,\ldots, a_t$ where $\gcd(k,p)=1$ and $j\geq 1$. Then if $\gcd(j,m)=1$, a distinct term, $rp^s$, such that $\gcd(r,p)=1$ and $m\nmid s$ must appear in the sequence. So $a_t\geq kp^j+p$. 
\end{lemma}
\begin{proof}
    Let $kp^j$ be the smallest multiple of $p$ in a $m$-product sequence, $0<a_1,\ldots,a_t$, where $\gcd(k,p)=\gcd(j,m)=1$. As there exists $R\in\mathbb{N}$ such that $\prod_{i=1}^t a_i=R^m$, we have that $p|R^m$ and so $p|R$. Therefore $p^s||R$ for some $s\geq 1$ and $p^{sm}||R^m$. Let $1\leq l \leq m-1$ be the number of times $kp^j$ appears in the sequence.
    
    Assume all terms other than $kp^j$ in the sequence are of the form $rp^{vm}$ for some $r,v\in\mathbb{N}$ and $\gcd(r,p)=1$. As $\prod_{i=1}^t a_i=R^m$, we see that $p^{jl+wm}||R^m$ for some $w\in\mathbb{N}$. Therefore $sm=jl+wm$ so $m|jl$ and as $\gcd(j,m)=1$ we see that $m|l$. However this is a contradiction, as $0<l<m$.
\end{proof}

\begin{corollary}\label{lemma:n+p}
    If $p$ is a prime such that $p^j||n$ where $j\geq 1$ and $\gcd(j,m)=1$ then $g_m(n)\geq n+p$. 
\end{corollary}
\begin{proof}
    All corresponding sequences for $g_m(n)$ must begin with $n$, so all of the sequences have $n$ as the smallest multiple of $p$ appearing in them. By Lemma \ref{lemma:p}, the largest term of all these sequences must be greater than or equal to $n+p$. Therefore $g_m(n)\geq n+p$.
\end{proof}

\begin{corollary}\label{lemma:2p}
    For all primes $p$, we have $g_m(p)\geq 2p$.
\end{corollary}
\begin{proof}
    This follows from Corollary \ref{lemma:n+p}
\end{proof}

We now consider upper bounds for $g_m(n)$. There have been significant results proven for $g_2(n)$. In \cite{Erdos}, Selfridge proved that, letting $P(n)$ is the largest prime factor of $n$, if $P(n)>\sqrt{2n}+1$ then $g_2(n)=n+P(n)$ and if $P(n)\leq \sqrt{2n}+1$ then $$g_2(n)\leq n+ \frac{3}{4}(\sqrt{8n+1}+1),$$ for $n\neq 2,3,8,10,32$. By considering $n=p(p-1)$ for some prime $p$, we see that as $$p(p-1)\cdot (p-1)(p+1)\cdot p(p+1)$$ is a square, $g_2(p(p-1))\leq p(p+1)$. By Lemma \ref{lemma:p}, we see that $g_2(p(p-1))\geq p(p+1)$. Therefore $g_2(n)=n+2p>n+2\sqrt{n}$. So the order of magnitude of Selfridge's bound cannot be improved. Additionally, if both $p$ and $2p+1$ are primes, then, letting $n=p(2p-1)$, as $$p(2p-1)\cdot 2p^2\cdot (p+1)(2p-1)\cdot p(2p+2)$$ is a square, we see that $g_2(p(2p-1))\leq p(2p+2)$. By Lemma \ref{lemma:p}, we see that either $p(2p+1)$ is in the sequence, or $g_2(p(2p-1))\geq p(2p+2)$. However if $p(2p+1)$ is in the sequence, then Lemma \ref{lemma:p} implies that $$g_2(p(2p-1))\geq (p+1)(2p+1)>p(2p+2).$$ Therefore $g_2(p(2p-1))\geq p(2p+2)$ and so $g_2(n)=n+3p=n+\frac{3}{4}(\sqrt{8n+1}+1)$. It follows that, if there are infinitely many primes $p$ such that $2p+1$ is a prime (a conjecture which has been widely studied \cite{dubner1996large}), then Selfridge's upper bound is tight.  

This inspired the authors in \cite{MR4706772} to investigate the relation between $g_2(n)-n$ and $P(n)$, leading them to show that, for any $c\in(0,1]$, $$\lim_{x\rightarrow \infty}\frac{\{n\leq x \,:\, g_2(n)-n<n^c\}}{x}=\lim_{x\rightarrow \infty}\frac{\{n\leq x \,:\, P(n)<n^c\}}{x}.$$ They also showed that for any $c>0$, there are infinitely many $n$ such that $g_2(n)-n<n^c$ and they provided lower bounds for $g_2(n)-n$ for non-square $n$.

It was stated without proof by Peter Kagey in \cite{citation-key} that for fixed $k$, we have $g_2(kp)=(k+1)p$ for sufficiently large primes $p$. The lower bound follows from Corollary \ref{lemma:n+p}.
We provide an upper bound for $g_m(kr)$ where $r$ is sufficiently large and not necessarily a prime.

 \begin{lemma}\label{lemma:r bound}
 For each $m\geq 2$ and $k\geq 1$, there exists $C(m,k)$ such that $g_m(kr)\leq (k+1)r$ for $r> C(m,k)$.
\end{lemma}
\begin{proof}
     If we can find two integers $t_1=kl_1^m$ and $t_2=(k+1)l_2^m$ where $kr<t_1,t_2<(k+1)r$, then $kr,t_1^{\{m-1\}},t_2,((k+1)r)^{\{m-1\}}$ forms an $m$-product sequence starting at $kr$. Such integers exist if there are some $l_1,l_2\in\mathbb{N}$ such that $r<l_1^m<\frac{k+1}{k}r$ and $\frac{k}{k+1}r<l_2^m<r$, or equivalently, $r^{\frac{1}{m}}<l<\left(\frac{k+1}{k}r\right)^{\frac{1}{m}}$ and $\left(\frac{k}{k+1}r\right)^{\frac{1}{m}}<l_2<r^{\frac{1}{m}}$. These $l_1,l_2$ is guaranteed to exist if both $$r^{\frac{1}{m}}\left(\left(\frac{k+1}{k}\right)^{\frac{1}{m}}-1\right)>1 \text{ and } r^{\frac{1}{m}}\left(1-\left(\frac{k}{k+1}\right)^\frac{1}{m}\right)>1.$$ These can be rearranged to $$r>\left(\left(\frac{k+1}{k}\right)^\frac{1}{m}-1\right)^{-m} \text{ and } r>\left(1-\left(\frac{k}{k+1}\right)^{\frac{1}{m}}\right)^{-m}.$$ 
     
     As $$\left(1-\left(\frac{k}{k+1}\right)^{\frac{1}{m}}\right)^{-m}\geq\frac{k+1}{k}\left(\left(\frac{k+1}{k}\right)^\frac{1}{m}-1\right)^{-m}>\left(\left(\frac{k+1}{k}\right)^\frac{1}{m}-1\right)^{-m},$$ if $r>\left(1-\left(\frac{k}{k+1}\right)^{\frac{1}{m}}\right)^{-m}:=C(m,k)$ we see that $g_m(kr)\leq (k+1)r$.
\end{proof}

The following lemma was shown in \cite{kagey2024conjecture}.

\begin{lemma}[{\cite[Lemma 6]{kagey2024conjecture}}]\label{lemma:2n}
   For all $n\geq 4$, we have $g_2(n)\leq 2n$. 
\end{lemma}

We will prove that an analogous statement holds for all $m$.

We prove this for $m=3$ first.

\begin{lemma}\label{lemma:3}
    We have $g_3(n)\leq 2n$ if and only if $n\neq 4$.
\end{lemma}
\begin{proof}
    By Lemma \ref{lemma:r bound}, we see that $C(3,1)<114$. Therefore $g_3(n)\leq 2n$ for all $n\geq 114$. In the Appendix, we have shown that $g_3(n)\leq 2n$ for $n\leq 7$ and $n\neq 4$. We also show that $g_3(4)=9>8$. It remains to show the result holds for $8\leq n\leq 113$.
    
    By Lemma \ref{lemma:m-th power}, we see that the result holds for $8$. For $9\leq n\leq 16$ we have $n\cdot 16\cdot (2n)^2=(4n)^3$. For $17\leq n\leq 31$ we have $n^2\cdot 32\cdot 2n=(4n)^3$. For $32\leq n\leq 54$ we have $n\cdot 54\cdot (2n)^2=(6n)^3$. For $55\leq n\leq 107$ we have $n^2\cdot 108\cdot 2n=(6n)^3$. For $108\leq n\leq 128$ we have $n\cdot 128\cdot (2n)^2=(8n)^3$. Therefore $g_3(n)\leq 2n$ if $8\leq n\leq 113$. 
    \end{proof}

We will now prove this for prime $m\geq 5$ and split the proof into cases based on whether $n$ is a power of $2$ or not.

\begin{lemma}\label{lemma:2^k}
    Let $p\geq 5$ be a prime. Then $g_p(2^k)\leq 2^{k+1}$ for all $k\geq 0$.
\end{lemma}
\begin{proof}
If $k\equiv 0\mod p$, then $g_p(2^k)=2^k<2^{k+1}$ by Lemma \ref{lemma:m-th power}.

If $k\equiv -1\mod p$, then let $k=sp-1$ for $s\geq 1$. It follows that $$2^{sp-1}<3^2\cdot 2^{sp-4}<3\cdot 2^{sp-2}<3^3\cdot 2^{sp-5}<2^{sp}.$$ Then as $(2^{sp-1})\cdot (3^2\cdot 2^{sp-4})\cdot (3\cdot 2^{sp-2})^{p-5}\cdot (3^3\cdot 2^{sp-5})$ is a $p$-th power, we see that $g_p(2^k)< 2^{k+1}$.

If $k\not\equiv 0\mod p$ and $k\not\equiv -1\mod p$, then let $k\equiv a\mod p$. As $(2^k)^{a+1}\cdot (2^{k+1})^{p-a}$ is a $p$-th power and $a+1\not\equiv 0\mod p$, we see that $g_p(2^k)\leq 2^{k+1}$.

\end{proof}

\begin{lemma}\label{lemma:not 2^k}
    Let $p\geq 5$ be a prime and $n$ not a power of $2$. Then $g_p(n)\leq 2n$.
\end{lemma}
\begin{proof}
    There exists a unique $k$ such that $n<2^k<2n$. 
    
    If $k\neq 0\mod p$, then letting $k\equiv a\mod p$, we see that $n^a\cdot 2^k\cdot (2n)^{p-a}$ is a $p$-th power. Therefore $g_p(n)\leq 2n$. 

    So let $k\equiv 0\mod p$, which implies that $k\geq p\geq 5$. Either $n<3\cdot2^{k-1}<2n$ or $n\leq 3\cdot2^{k-2}<2n$. Similarly, either $n<3^3\cdot2^{k-4}<2n$ or $n\leq3^3\cdot2^{k-5}<2n$. 

    \textbf{Case 1: $n<3\cdot 2^{k-1}<3^3\cdot2^{k-4}<2n$.}

    Let $a\equiv p\mod 3$, with $a\in\{1,2\}$ and $b\equiv \frac{a-p}{3}\mod p$ with $b\in\{1,2,\ldots,p\}$. As $a\not\equiv 0\mod p$, we see that $b\equiv\frac{a-p}{3}\not \equiv 0\mod p$. Then $n^b\cdot (3\cdot 2^{k-1})^a\cdot (3^3\cdot 2^{k-4})^{\frac{p-a}{3}}\cdot (2n)^{p-b}$ is a $p$-th power, so it follows that $g_p(n)\leq 2n$.

    \textbf{Case 2: $n\leq3^3\cdot2^{k-5}<3\cdot 2^{k-1}<2n$.}

    Let $a\equiv p\mod 3$, with $a\in\{1,2\}$ and $b\equiv \frac{2a-2p}{3}\mod p$ with $b\in\{1,2,\ldots,p\}$. Similarly to Case 1, as $a\not\equiv 0\mod p$ we have that $b\equiv \frac{2a-2p}{3}\not\equiv 0\mod p$ and $\frac{p-a}{3}+b\equiv\frac{a-p}{3}\not\equiv 0\mod p$. Then $n^b\cdot (3^3\cdot 2^{k-5})^{\frac{p-a}{3}}\cdot (3\cdot 2^{k-1})^a\cdot (2n)^{p-b}$ is a $p$-th power, so it follows that $g_p(n)\leq 2n$.

    \textbf{Case 3: $n\leq3\cdot 2^{k-2}<3^3\cdot2^{k-4}<2n$.}

    This is not possible as $n\leq 3\cdot 2^{k-2}$ implies that $2n\leq 3\cdot 2^{k-1}<3^3\cdot2^{k-4}<2n$ which is a contradiction.

    \textbf{Case 4: $n\leq3\cdot 2^{k-2}<3^3\cdot2^{k-5}<2n$.}

    Let $a\equiv p\mod 3$, with $a\in\{1,2\}$ and $b\equiv -\frac{a+2p}{3}\mod p$ with $b\in\{1,2,\ldots,p\}$. Similarly to above, as $a\not\equiv 0\mod p$ we see that $b\equiv -\frac{a+2p}{3}\not\equiv 0\mod p$ and $a+b\equiv \frac{2a-2p}{3}\not\equiv 0\mod p$. Then $n^b\cdot (3\cdot 2^{k-2})^a\cdot (3^3\cdot 2^{k-5})^{\frac{p-a}{3}}\cdot (2n)^{p-b}$ is a $p$-th power, so it follows that $g_p(n)\leq 2n$.
\end{proof}

We will use the following lemma to prove this for composite $m$.

\begin{lemma}\label{lemma:divisible}
    If $l|m$, then $g_m(n)\leq g_l(n)$.
\end{lemma}
\begin{proof}
    If $g_l(n)=k$, then there exists a corresponding sequence for $g_l(n)$, namely $n=a_1\leq \ldots\leq a_t=k$. Then consider the sequence $n=a_1^{\{\frac{m}{l}\}}, a_2^{\{\frac{m}{l}\}},\ldots,a_t^{\{\frac{m}{l}\}}=k$ and label the terms of the sequence $n=b_1\leq \ldots \leq b_{t'}=k$. We claim that this is a $m$-product sequence starting at $n$ and ending at $k$, therefore showing that $g_m(n)\leq k=g_l(n)$. First note that all the integers in the sequence are contained in $[n,k]$ as they are the same integers in the corresponding sequence for $g_l(n)$. Additionally, $\prod_{i=1}^{t'} b_i=(\prod_{i=1}^t a_i)^{\frac{m}{l}}=(R^l)^{\frac{m}{l}}=R^m$ for some $R\in\mathbb{N}$. Finally, as each integer appeared at most $l-1$ times in $a_1,\ldots,a_t$, we see that each integer appears at most $(l-1)\frac{m}{l}<m$ times in $b_1,\ldots,b_{t'}$. 
\end{proof}

On a side note, as $C(m,k)=\left(1-\left(\frac{k}{k+1}\right)^{\frac{1}{m}}\right)^{-m}$ from Lemma \ref{lemma:r bound} is a strictly increasing function in $m$, we can use Lemma \ref{lemma:divisible} to achieve a smaller bound for composite $m$. Letting $p$ be the smallest prime that divides $m$, for $n>C(p,k)$ we have that $g_m(kn)\leq g_p(kn)\leq (k+1)n$.

Lemma \ref{lemma:divisible} allows us to prove the following generalisation to Lemma \ref{lemma:2n}. 
\begin{theorem}\label{theorem:upper}
    For $m\geq 2$ and $n\in\mathbb{N}$, $g_{m}(n)\leq 2n$ if and only if $(m,n)\neq\{(2,2),(2,3),(3,4)\}$.
\end{theorem}
\begin{proof}
    Let $p$ be the smallest prime that divides $m$. For all $n\geq 5$ we have $g_{m}(n)\leq g_p(n)\leq 2n$ by Lemmas \ref{lemma:2n}, \ref{lemma:3}, \ref{lemma:2^k}, \ref{lemma:not 2^k} and \ref{lemma:divisible}. We show in the Appendix that $g_{m}(0)=0$ and $g_{m}(1)=1$ for all $m\geq 2$, that $g_{m}(2)=4$ and $g_{m}(3)=6$ for $m\geq 3$ and that $g_m(4)\leq 8$ for $m\neq 3$.

    The Appendix also includes the proofs that $g_2(2)=6, g_2(3)=8$ and $g_3(4)=9$.
\end{proof}

\begin{corollary}\label{corollary:p}
    For any $m\geq 2$ and prime $p$, we have $g_m(p)=2p$ if and only if $(m,p)\neq \{(2,2),(2,3)\}$.
\end{corollary}

\begin{proof}
    This follows from Theorem \ref{theorem:upper} and Corollary \ref{lemma:2p}. 
\end{proof}

Theorems \ref{conj} and \ref{theorem:upper}, and calculations of $g_m(n)$ for small $n$, leads us to conjecture the following:

\begin{conjecture}
For all $m\geq 2$ and all $n\in\mathbb{N}$, we have $g_m(n)=2n$ if and only if:
\begin{enumerate}
    \item $n$ is prime and $(m,n)\neq\{(2,2),(2,3)\}$

    \item $n=0$

    \item $n=4$ for odd $m>3$ 

    \item $(m,n)=\{(2,6),(3,6),(6,6)\}$.
\end{enumerate}
\end{conjecture}

We have shown the if direction of the conjecture holds by combining Corollary 4.11 with the explicit small cases covered in the Appendix

We can prove that the full conjecture holds for $2m$ when $m\geq 2$.

\begin{theorem}
    For all $m\geq 2$ and all primes $p$, we have $g_{2m}(p)=2p$ if and only if $n$ is prime, $n=0$ or $(2m,n)=(6,6)$.
\end{theorem}
\begin{proof}
    We only need to show that for all $n$ not mentioned in the statement, $g_{2m}(n)< 2n$.
    
    For any $n\geq 4$ not equal to a prime, $0$ or $6$, by Theorem \ref{conj} and Lemmas \ref{lemma:2n} and \ref{lemma:divisible} we see that $g_{2m}(n)\leq g_2(n)<2n$. We show that $g_m(1)=1$ for all $m$ and that $g_{2m}(6)=9$ when $2m\neq 6$ in the Appendix.
\end{proof}

\section{Bijectivity of \texorpdfstring{$g_m(n)$}{gm(n)}}
We prove that $g_m(n)$ is a bijection between $\mathbb{N}$ and $\mathbb{N}\backslash\mathbb{P}$ when $m$ is prime by generalising the method outlined in $\cite{kagey2024conjecture}$ for $m=2$, which follows Michael Reid's proof \cite{efdab2f9-940e-3325-88c5-7e0949aa0d3d}. 
\begin{theorem}\label{theorem:injective}
The function $g_m(n)$ is injective for all prime $m$.
\end{theorem}
\begin{proof} 
First, note that as $n\leq g_m(n)$, there cannot exist non-zero $n\in\mathbb{N}$ such that $g_m(n)=0$.

Let $0<n<n'$ and assume $g_m(n)=g_m(n')$. Then there exists a corresponding sequence for $g_m(n)$, $$n=a_1\leq \ldots\leq a_t=g_m(n),$$ where $\prod_{i=1}^t a_i=R^m$ for some $R\in\mathbb{N}$.

There similarly exists a corresponding sequence for $g_m(n')$, $$n'=b_1\leq\ldots\leq b_{t'}=g_m(n')=g_m(n),$$ where $\prod_{i=1}^{t'}b_i=S^m$ for some $S\in\mathbb{N}$.

Let these sequences be $A$ and $B$ respectively. Let $n=a_1$ appear $1\leq s_1\leq m-1$ times in $A$ and $g_m(n)=a_t$ appear $1\leq s_t\leq m-1$ times in $A$. Let $g_m(n)=b_{t'}$ appear $1\leq v_{t'}\leq m-1$ times in $B$. 

Now consider the $m$-product sequence, $C_k=A\cdot B^k$ for $1\leq k\leq m$, whose terms are $c_1^k\leq\ldots\leq c_{s}^k$. Then, by Corollary \ref{lemma:k product}, $c_1^k=n$ appears $s_1$ times in $C_k$, $c_s^k\leq g_m(n)$ and $g_m(n)$ appears $s_t+kv_{t'}\mod m$ times in $C$.

As $m$ is prime and $1\leq v_{t'}\leq m-1$, we see that $kv_{t'}\equiv lv_{t'}\mod m$ if and only if $k\equiv l\mod m$. So for $k,l\in[1,m]$ and $k\neq l$, it is clear that $s_t+kv_{t'}\not\equiv s_t+lv_{t'}\mod m$. By the pigeonhole principle, there must then exist a $1\leq r\leq m$ such that $s_t+rv_{t'}\mod m = 0$. However then there is a $m$-product sequence, $C_r=A\cdot B^r$, starting at $n$ whose largest term is less than or equal to $g_m(n)$ and $g_m(n)$ appears $s_t+rv_{t'}\mod m = 0$ times in the sequence. Therefore $g_m(n)$ does not appear in the sequence and the largest term of the $m$-product sequence starting at $n$ is strictly less than $g_m(n)$. This contradicts the definition of $g_m(n)$ and so $g_m(n)$ is injective. 
\end{proof}

We split the proof of surjectivity over several lemmas. 

\begin{lemma}\label{lemma:range}
For all $m\geq 2$ and $n\in\mathbb{N}$, we have $g_m(n)\in \mathbb{N}\backslash \mathbb{P}$.
\end{lemma}
\begin{proof}
    If $g_m(n)=p$ then all terms of any corresponding sequence of $g_m(n)$ are positive and less than or equal to $p$. However this contradicts Lemma \ref{lemma:p}.
\end{proof}

Define the function $\overline{g_m}:\mathbb{N}\rightarrow \mathbb{N}$ such that $\overline{g_m}(n)$ is the greatest integer, $k$, such that there exists a $m$-product sequence starting at $k$ with largest term equalling $n$, $$k=a_1\leq \ldots\leq a_t=n.$$

\begin{lemma}\label{lemma:overline}
    If $n\in\mathbb{N}\backslash\mathbb{P}$, $n\neq 0$, then $\overline{g_m}(n)>0$.
\end{lemma}
   \begin{proof}
      Clearly $\overline{g_m}(1)=1$, so consider $n>1$ composite. Then $n=p_1^{e_{p_1}}\ldots p_r^{e_{p_r}}$ and let $q_1<q_2<\ldots<q_l$ be the primes in the factorisation such that $m\nmid e_{q_i}$ for $1\leq i\leq l$. Let $e_{q_i}\equiv s_{q_i} \mod m$ and note that $s_{q_i}$ is not $0$ for any $i$. Then $n\prod_{i=1}^l q_i^{M-s_{q_i}}$ is a $m$-th power. Note that $q_i<n$ for all $1\leq i\leq k$. Therefore the sequence $$q_1^{\{M-s_{q_1}\}},q_2^{\{M-s_{q_2}\}},\ldots,q_i^{\{M-s_{q_i}\}},\ldots,q_l^{\{M-s_{q_l}\}},n$$ is a $m$-product sequence as $M-s_{q_i}<M$ for all $1\leq i\leq l$, so we see that $\overline{g_m}(n)\geq q_1>0$.
   \end{proof} 

   \begin{lemma}\label{lemma:inverse}
       Let $m\in\mathbb{P}$, $n\in\mathbb{N}\backslash\mathbb{P}$, $n\neq 0$. Then $g_m(\overline{g_m}(n))= n$. 
   \end{lemma}
    \begin{proof}
    It is clear that $0<g_m(\overline{g_m}(n))\leq n$ as by the definition of $\overline{g_m}(n)$ and Lemma \ref{lemma:overline}, there exists a $m$-product sequence, $A$, starting at $\overline{g_m}(n)$, $$0<\overline{g_m}(n)=a_1\leq \ldots \leq a_t=n.$$
    
    We will show that $g_m(\overline{g_m}(n))=n$. Assume that $g_m(\overline{g_m}(n))<n$. Then there exists a $m$-product sequence, $B$, starting at $\overline{g_m}(n)$, $$0<\overline{g_m}(n)=b_1\leq \ldots \leq b_{t'}<n$$. Similarly to the proof of the injectivity of $g_m$, let $1\leq r\leq m-1$ be the number of times $n$ appears in $A$, let $1\leq s\leq m-1$ be the number of times $\overline{g_m}(n)$ appears in $A$ and let $1\leq v\leq m-1$ be the number of times $\overline{g_m}(n)$ appears in $B$. 

    Now consider the product $\prod_{i=1}^t a_i\prod_{j=1}^{t'}b_j$. This is a non-zero $m$-th power as both $\prod_{i=1}^t a_i$ and $\prod_{j=1}^{t'}b_j$ are non-zero $m$-th powers. If an integer, $x$, appears $h$ times in $\{a_1,\ldots,a_t,b_1,\ldots,b_{t'}\}$, then remove $\lfloor\frac{h}{m}\rfloor m$ terms in the sequence that equal $x$. As this is a multiple of $m$, the product of the remaining terms is still an $m$-th power. We needed $\overline{g_m}(n)>0$ for this part as otherwise, if $0$ appeared $w$ times where $m|w$, removing all $0$s from the sequence would not guarantee that the product of the remaining terms was still an $m$-th power.
    
    Label the remaining terms $c_1,\dots,c_{t_1}$ such that $c_1\leq\ldots\leq c_{t_1}$. Denote the sequence $C$. It is clear that any integer, $x$, appears in $C$ less than $m$ times as otherwise at least $m$ further terms equalling $x$ would have been removed. Note that $\overline{g_m}(n)$ is the smallest integer in $A$ and the smallest integer in $B$, so $c_1\geq\overline{g_m}(n)$ and in fact, $\overline{g_m}(n)$ appears $s+v\mod m$ times in $C$. As $n$ is the largest term in $A$, and does not appear in $B$, we have that $c_{t_1}=n$ and $n$ appears in $C$ exactly $r$ times. Therefore $C$ is a $m$-product sequence starting at $c_1\geq\overline{g_m}(n)$ and the greatest integer in $C$ is $n$.

    We can replace $A$ by $C$ and repeat this process, creating a $m$-product sequence, $d_1,\ldots,d_{t_2}$ where $n=d_{t_2}$ appears $r$ times, $d_1\geq \overline{g_m}(n)$ and $\overline{g_m}(n)$ appears $s+2v\mod m$ times. By repeating this $1\leq k\leq m$ times in total, we will get a $m$-product sequence with these properties except that $\overline{g_m}(n)$ appears $s+kv\mod m$ times. As $m$ is prime and $1\leq v_\leq m-1$, we see that $kv\equiv lv\mod m$ if and only if $k\equiv l\mod m$. So for $k,l\in[1,m]$ and $k\neq l$, it is clear that $s+kv\not\equiv s+lv\mod m$. By the pigeonhole principle, there must then exist a $1\leq j\leq m$ such that $s+jv\mod m = 0$. However then there is a $m$-product sequence with largest term $n$ and smallest term greater than or equal to $\overline{g_m}(n)$ where $\overline{g_m}(n)$ appears $s+jv\mod m = 0$ times. Therefore $\overline{g_m}(n)$ does not appear in the sequence and the smallest term of the $m$-product sequence is strictly greater than $\overline{g_m}(n)$. This contradicts the definition of $\overline{g_m}(n)$ so $g_m(\overline{g_m}(n))=n$.
    \end{proof}

    \begin{theorem}\label{theorem:surjective}
    The function $g_m(n)$ is surjective on $\mathbb{N}\backslash\mathbb{P}$ for all prime $m$.
    \end{theorem}
    \begin{proof}
    For any $n\in\mathbb{N}\backslash\mathbb{P}$, $n>0$, we see that for $k=\overline{g}(n)$, we have $g(k)=n$ by Lemma \ref{lemma:inverse}. Additionally, $g_m(0)=0$. It follows that $g_m(n)$ is surjective on $\mathbb{N}\backslash\mathbb{P}$.
    \end{proof}

We believe that the converse is also true.

\begin{conjecture}\label{conjecture:composite}
     For all composite $m$, $g_m(n)$ is neither injective nor surjective.
\end{conjecture}

The following results almost entirely resolve whether $g_m(n)$ can be injective for composite $m$. 

\begin{theorem}\label{Theorem:composite}
    Let $p$ be a prime such that $p|m$. If there exists $k$ such that $g_m(k)\neq g_p(k)$, then $g_m(n)$ is not injective.
\end{theorem}
\begin{proof}
    Note that $g_m(n)\in\mathbb{N}\backslash\mathbb{P}$ for all $n\in\mathbb{N}$ by Lemma \ref{lemma:range}. Also, $g_p(n)$ is injective and surjective on $\mathbb{N}\backslash\mathbb{P}$ by Theorems \ref{theorem:surjective} and \ref{theorem:injective}.
    
    Then assume $g_m(n)$ is injective. 

    Let $g_m(k)\neq g_p(k)$. Then by Lemma \ref{lemma:divisible}, we see that $g_m(k)<g_p(k)$. As $g_p(n)$ is surjective, there exists $n_1\in \mathbb{N}$ such that $g_p(n_1)=g_m(k)<g_p(k)$. It follows that $n_1\neq k$ and so using the injectivity of $g_m(n)$, we see that $g_m(n_1)\neq g_m(k)$. With Lemma \ref{lemma:divisible} this implies that $g_m(n_1)<g_p(n_1)=g_m(k)$. 
    
    As $g_p(n)$ is surjective, there exists $n_2\neq n_1$ such that $g_p(n_2)=g_m(n_1)<g_p(n_1)$. By Lemma \ref{lemma:divisible}, we see that either $g_m(n_2)=g_p(n_2)=g_m(n_1)$, in which case $g_m(n)$ is not injective, or $g_m(n_2)<g_p(n_2)=g_m(n_1)$. So the latter case must hold.
    
    As $g_p(n)$ is surjective, there exists $n_3\not \in \{n_1, n_2\}$ such that $g_p(n_3)=g_m(n_2)<g_p(n_2)<g_p(n_1)$. By Lemma \ref{lemma:divisible}, we see that either $g_m(n_3)=g_p(n_3)=g_m(n_2)$, in which case $g_m(n)$ is not injective, or $g_m(n_3)<g_p(n_3)=g_m(n_2)$. So the latter case must be true.

    As $g_p(n)$ is surjective, there exists $n_4\not \in \{n_1, n_2,n_3\}$ such that $g_p(n_4)=g_m(n_3)<g_p(n_3)<g_p(n_2)<g_p(n_1)$. By Lemma \ref{lemma:divisible}, we see that either $g_m(n_4)=g_p(n_4)=g_m(n_3)$, in which case $g_m(n)$ is not injective, or $g_m(n_4)<g_p(n_4)=g_m(n_3)$. So the latter case is true.
    
    We can repeat this and generate an infinite sequence of distinct terms $n_1,n_2,n_3\ldots,$ such that $g_p(n_1)>g_p(n_2)>g_p(n_3)>\ldots$ and $g_m(n_i)<g_p(n_i)$ for all $i$. However as $g_p(n)\in\mathbb{N}$ for all $n\in\mathbb{N}$, we see that $g_p(n)\geq 0$ for all $n$ and $g_p(n_i)-1\geq g_p(n_{i+1})$. Then there must exist $i\leq g_p(n_1)+1$ such that $g_p(n_i)=0$. So $g_p(n_{i+1})\geq g_p(n_i)=0$, which is a contradiction. Therefore $g_m(n)$ is not injective. 
\end{proof}

With Theorem \ref{theorem:injective}, this implies that if $p|m$ then $g_m(n)$ is injective if and only if $g_m(n)=g_p(n)$ for all $n\geq 0$. We show that this cannot be the case if $m$ has more than one prime factor.  

\begin{corollary}
    If $m$ is not a prime power, then $g_m(n)$ is not injective.
\end{corollary}
\begin{proof}
    Let $p$ and $q$ be distinct primes such that $p|m$ and $q|m$. Let $p^k||m$ for $k\geq 1$. Then for any prime $r$, it follows from Lemma \ref{lemma:m-th power} that $g_m(r^q)=r^{q}<g_p(r^q)$. Then Theorem \ref{Theorem:composite} implies that $g_m(n)$ is not injective.
\end{proof}

We can also show that if $g_m(n)$ is injective, then it is surjective. We do this by proving the contrapositive.

\begin{corollary}
    If $g_m(n)$ is not surjective on $\mathbb{N}\backslash\mathbb{P}$ then $g_m(n)$ is not injective.
\end{corollary}
\begin{proof}
    First note that $g_m(n)\in\mathbb{N}\backslash\mathbb{P}$ for all $n\in\mathbb{N}$ by Lemma \ref{lemma:range}.

    If $g_m(n)$ is not surjective, there exists some $n'\in\mathbb{N}\backslash\mathbb{P}$ for which there is no $n\in\mathbb{N}$ such that $g_m(n)=n'$. Then there is a prime $p$ such that $p|m$ and $g_p(n)$ is injective and surjective on $\mathbb{N}\backslash\mathbb{P}$ by Theorems \ref{theorem:surjective} and \ref{theorem:injective}. Therefore there exists $n_1$ such that $g_p(n_1)=n'$ and as there is no $n\in\mathbb{N}$ such that $g_m(n)=n'$, it follows from Lemma \ref{lemma:divisible} that $g_m(n_1)<g_p(n_1)=n'$. By Theorem \ref{Theorem:composite}, we see that $g_m(n)$ is not injective.
\end{proof}

It would follow from Theorem \ref{theorem:injective} and Conjecture \ref{conjecture:composite} that if $g_m(n)$ is surjective, then it is injective.

We achieve some partial results towards the full conjecture by calculating $g_m(n)$ for $0\leq n\leq 64$ and $m\geq 2$. We explain the calculations needed to prove these results in the Appendix. 

Let $v(m)$ equal the smallest prime that divides $m$.

\begin{lemma}\label{lemma:check}
    If $m$ is composite and $v(m)=2,3,5,7,11$ or $17$ then $g_m(n)$ is not surjective or injective.
\end{lemma}

The smallest composite $m$ for which we do not know whether $g_m(n)$ is injective or surjective is $13^2$.

These functions are often not injective on squares and higher powers. It appears that for any prime $p$ and $k\geq 1$, if $s$ is not a $p$-th power then there exists $r\neq s^{p^k}$ such that $g_{p^{k+1}}(r)=s^{p^k}=g_{p^{k+1}}(s^{p^k})$. We have verified this for $s^{p^k}\leq 128$. The case where $s=2$ and $k=1$ is the following conjecture:

\begin{conjecture}
    For any prime $p$, there exists $r\neq 2^p$ such that $g_{p^2}(r)=2^p$.
\end{conjecture}

This conjecture would imply that $g_{p^k}(n)$ is not injective for $k>1$. This follows as $g_p(2^p)=2^p$ by Lemma \ref{lemma:m-th power}, so Theorem \ref{theorem:injective} implies that $g_p(r)\neq 2^p$ for $r\neq 2^p$. As $g_{p^2}(r)\neq g_p(r)$, we see that $g_{p^2}(r)<g_p(r)$ by Lemma \ref{lemma:divisible} and so $g_{p^k}(r)\leq g_{p^2}(r)<g_p(r)$. Therefore, Theorem \ref{Theorem:composite} implies that $g_{p^k}(n)$ is not injective.

For each composite $m$, we expect there to be an infinite number of distinct pairs $n,n'$ such that $g_m(n)=g_m(n')$ and an infinite number of $x\in\mathbb{N}$ such that $g_m(n)\neq x$ for all $n\in\mathbb{N}$. 

We can generate an infinite number of pairs $n,n'$ such that $g_{4t}(n)=g_{4t}(n')$, for $t\geq 1$, if we assume that there are infinitely many primes which are one more than a square. This is a well known conjecture, proposed as Conjecture $E$ in \cite{hardy1923some}. 

\begin{proposition}
    There an infinite number of distinct pairs, $n,n'\in\mathbb{N}$, such that $g_{4t}(n)=g_{4t}(n')$ if there exist an infinite number of primes of the form $a^2+1$ for $a\in\mathbb{N}$.
\end{proposition}
\begin{proof}
    Choose $0<x\in\mathbb{N}$ such that $x^2+1$ is prime, and let $p=x^2+1$. Then by Lemma \ref{lemma:n+p}, $g_{4t}(x^2(x^2+1))\geq x^2(x^2+1)+x^2+1=(x^2+1)^2$. It can be checked that $(x^2(x^2+1))^{\{2t\}},((x^2+1)^2)^{\{t\}}$ is a $4t$-product sequence starting at $x^2(x^2+1)$ and so $g_{4t}(x^2(x^2+1))=(x^2+1)^2$. It follows from Lemma \ref{lemma:m-th power} that $g_{4t}((x^2+1)^2)=(x^2+1)^2$. Therefore $g_{4t}(p(p-1))=g_{4t}(p^2)$.
\end{proof}

If $g_m(n_1)=g_m(n_2)=\ldots=g_m(n_k)$ for distinct $n_i$, then we can also ask about how large $k$ can be for a given $m$. This is essentially asking how far from injective $g_m(n)$ can be. For instance, $g_m(18)=g_m(20)=g_m(25)=25$ if $4|m$ and $5\nmid m$ and $m>4$. Another example $g_m(20)=g_m(24)=g_m(27)=27$ if $9|m$ and $4\nmid m$ and $m>18$. We can provide some bounds for how large $k$ can be for a given $m$.

\begin{proposition}\label{lemma:k(m)}
    Let $m=p_1^{e_1}\cdot\ldots\cdot p_s^{e_s}$ and $k(m)$ be the largest integer for which there exists $n_1<n_2<\ldots<n_{k(m)}$ such that $g_m(n_1)=g_m(n_2)=\ldots=g_m(n_{k(m)})$. Then $k(m)< \prod_{i=1}^s (e_i+1)=d(m)$ where $d(m)$ counts the number of divisors of $m$.
\end{proposition}
\begin{proof}
   Take distinct $n_1<\dots<n_{k(m)}\in\mathbb{N}$ such that $g_m(n_1)=\ldots=g_m(n_{k(m)})$. If $n_1=0$ then $g_m(n_1)=0$ and as $n_i\leq g_m(n_i)=0$, we would have $n_i=0$ for all $1\leq i\leq k(m)$. As the $n_i$ are meant to be distinct, we must have $k(m)=1$. As $m\geq 2$, we have $d(m)\geq 2$ and so $k(m)<d(m)$.
   
   Take distinct $0<n_1<\dots<n_{k(m)}\in\mathbb{N}$ such that $g_m(n_1)=\ldots=g_m(n_{k(m)})$. For $1\leq i\leq k(m)$, there exists a corresponding sequence, $A_i$, for $g_m(n_i)$ with largest term $a_{t_i}$ which appears $1\leq r_i\leq m-1$ times in its sequence. Then we claim that $\gcd(r_i,m)\nmid\gcd(r_{j},m)$ for $1\leq j<i\leq k(m)$. Note that $\gcd(r_i,m)\neq m$ as $1\leq r_i\leq m-1$ for any $1\leq i\leq k(m)$, so $\gcd(r_i,m)< m$.

    Let $r_i=\gcd(r_i,m)t$ where $\gcd(t,\frac{m}{\gcd(r_i,m)})=1$. Let $t^{-1}$ be $s\in\{1,\ldots,\frac{m}{\gcd(r_i,m)}-1\}$ such that $t\cdot s\mod m\equiv 1\mod \frac{m}{\gcd(r_i,m)}$. 
    
    Then if $\gcd(r_i,m)|\gcd(r_j,m)$ for $1\leq j<i\leq k(m)-1$ and letting $q\equiv -\frac{r_j}{\gcd(r_i,m)}t^{-1}\mod \frac{m}{\gcd(r_i,m)}$ such that $q\in \{0,1,\ldots,\frac{m}{\gcd(r_i,m)}-1\}$, we see by Lemma \ref{lemma:k product} that $$A_j\cdot A_i^{q},$$ is a $m$-product sequence with smallest term $n_j$ and largest term less than or equal to $g_m(n_j)$. However $g_m(n_j)$ appears
    $$r_j+q\cdot r_i\mod m = 0$$ times in the sequence. This means that the largest term of an $m$-product sequence, with smallest term $n_j$, is strictly less than $g_m(n_j)$. This contradicts the definition of $g_m(n_j)$.

    As $\gcd(r_i,m)\nmid\gcd(r_{j},m)$ for $1\leq j<i\leq k(m)-1$, it follows that $\gcd(r_i,m)\neq \gcd(r_j,m)$ for $i\neq j$. It is clear that $\gcd(r_i,m)$ is a divisor of $m$ for each $1\leq i\leq k(m)$ so there can be at most $d(m)$ distinct values of $\gcd(r_i,m)$. However we cannot have $\gcd(r_i,m)=m$. So there can be at most $d(m)-1$ distinct values of $\gcd(r_i,m)$ and therefore $k(m)<d(m)$. 
\end{proof}

Note that as $g_8(18)=g_8(20)=g_8(25)=25$ and $g_{27}(20)=g_{27}(24)=g_{27}(27)=27$, the above bound is tight for $m=4$ and $m=8$. Additionally, if $m$ is a prime then it is clear that $k(m)=1$, so the bound is tight for prime $m$ as well.

\begin{conjecture}
   The bound in Proposition \ref{lemma:k(m)} is tight for all prime powers $m$, that is $k(m)=d(m)-1$.
\end{conjecture}

\section{\texorpdfstring{$T_m(n)$}{Tm(n)} and \texorpdfstring{$T'_m(n)$}{T'm(n)}}
We begin by introducing the definition of the length of a sequence, as defined in \cite{kagey2024conjecture}.

   Let $A=(a_i)_{i=1}^t$ be a product sequence, then we say $L(A)=t$ is the number of terms in $A$. We call this the \textit{length} of $A$.

    Let $C_m(n)$ be the set of corresponding sequences for $g_m(n)$, then we define $$T_m(n)=\min_{A\in C_m(n)} L(A),$$ to be the minimum length of a corresponding sequence for $g_m(n)$.

\begin{definition}
    Let $A$ be a corresponding sequence for $g_m(n)$, $$a_1,\ldots,a_t.$$ Then we can also represent $A$ as $$b_1^{\{r_1\}},\ldots,b_{s}^{\{r_s\}},$$ where $s$ is the number of distinct integers that appear in $A$. We say $L'(A)=s$ is the distinct length of $A$.

    Let $C_m(n)$ be the set of corresponding sequences for $g_m(n)$, then we define $$T'_m(n)=\min_{A\in C_m(n)} L'(A),$$ to be the minimum distinct length of corresponding sequence for $g_m(n)$. Note that $T_2(n)=T'_2(n)$.
\end{definition}

There are a number of basic properties that we can outline. 
We state the following lemma without proof.
\begin{observation}\label{lemma:clear}
For all $m\geq 2$ and $n$, we have $T'_m(n)\leq T_m(n)$.
\end{observation}
An analogue of Lemma \ref{lemma:divisible} holds for $T'_m(n)$ using a similar proof. 
\begin{lemma}\label{lemma:T bound}
    If $l|m$ and $g_l(n)=g_m(n)$ then $T'_m(n)\leq T'_l(n)$.
\end{lemma}

\begin{proof}
    Let $T'_l(n)=k$, then there exists a corresponding sequence for $g_l(n)$, $$n=a_1^{\{e_1\}},a_2^{\{e_2\}},\ldots,a_k^{\{e_k\}}=g_l(n),$$ where each $e_i\leq l-1$. Then consider the sequence where each term is copied $\frac{m}{l}$ times, $$n=a_1^{\{e_1\frac{m}{l}\}},a_2^{\{e_2\frac{m}{l}\}},\ldots,a_k^{\{e_k\frac{m}{l}\}}=g_m(n).$$ As $\prod_{i=1}^k a_i^{e_i}$ is an $l$-th power, we see that $\prod_{i=1}^k a_i^{e_i\frac{m}{l}}=(\prod_{i=1}^k a_i^{e_i})^\frac{m}{l}$ is an $m$-th power. Therefore this is a corresponding sequence for $g_m(n)$ and $T'_m(n)\leq k$.
\end{proof}

The following lemmas provide a precise characterization about which $n$ satisfy $T'_m(n)=1$ and what values $T_m(n)$ can be.

\begin{lemma}\label{lemma:prime1}
    We have $T'_m(n)=1$ if and only if $g_m(n)=n$.
\end{lemma} 
\begin{proof}
    Assume $g_m(n)=n$. Then as every term, $a_i$, in a corresponding sequence for $g_m(n)$ satisfies $n\leq a_i\leq g_m(n)$, we see that every term in the sequence equals $n$. Therefore $T'(n)=1$.

    Assume $T'_m(n)=1$. Then there is only one integer that appears in the corresponding sequence for $g_m(n)$, and as both $n$ and $g_m(n)$ must appear in the sequence it follows that $g_m(n)=n$.
\end{proof} 

\begin{lemma}\label{lemma:prime2}
    Let $n=\prod_{i=1}^s p_i^{e_i}$ and $d=\gcd(e_1,\ldots,e_s)$. Then $T_m(n)=\frac{m}{\gcd(d,m)}$ and $T'_m(n)=1$ if and only if $\gcd(d,m)>1$. 
\end{lemma}
\begin{proof}
    Note that $n=k^d$ for $k=\prod_{i=1}^s p_i^{\frac{e_i}{d}}$. Then if $\gcd(d,m)>1$, by Lemma \ref{lemma:m-th power} we see that $g_m(n)=n$. From Lemma \ref{lemma:prime1} we see that $T'_m(n)=1$. Then $T_m(n)$ is the smallest $0<r\in\mathbb{N}$ such that $n^{\{r\}}$ is a corresponding sequence for $g_m(n)$. This means that it is the smallest $r$ such that $n^r$ is a $m$-th power. Clearly $$n^{\frac{m}{\gcd(d,m)}}=k^{\frac{dm}{\gcd(d,m)}}=k^{m\frac{d}{\gcd(d,m)}},$$ is a $m$-th power, so $r\leq \frac{m}{\gcd(d,m)}$. Therefore $m|re_i$ for every $1\leq i\leq s$ and so $$m|\gcd(re_1,re_2,\ldots,re_s)=r\cdot\gcd(e_1,\ldots,e_s).$$ So $m|rd$. It follows that $\frac{m}{\gcd(d,m)}|r\frac{d}{\gcd(d,m)}$ and as $\gcd(\frac{m}{\gcd(d,m)},\frac{d}{\gcd(d,m)})=1$ we see that $\frac{m}{\gcd(d,m)}|r$. As $r>0$, this implies that $r\geq \frac{m}{\gcd(d,m)}$. Therefore $r=\frac{m}{\gcd(d,m)}$. 

    If $T'_m(n)=1$ and $T_m(n)=\frac{m}{\gcd(d,m)}$ then by Lemma \ref{lemma:prime1} we see that $g_m(n)=n$ and so $n^{\{\frac{m}{\gcd(d,m)}\}}$ is a corresponding sequence for $g_m(n)$. Therefore $\frac{m}{\gcd(d,m)}<m$ and so $\gcd(d,m)>1$. 
\end{proof}

We now show that when $T'_m(n)=1$, these are the only possible values of $T_m(n)$.

\begin{lemma}\label{lemma:prime3}
    For a given $m\geq 2$, there exists $n$ such that $T'_m(n)=1$ and $T_m(n)=r$ if and only if $r|m$ and $r<m$. These $n$ are exactly those such that $n=\prod_{i=1}^s p_i^{e_i}$ where $\gcd(e_1,\ldots,e_s,m)=\frac{m}{r}>1$.
\end{lemma} 
\begin{proof}
    If $T'_m(n)=1$ and $T_m(n)=r$ then by Lemma \ref{lemma:prime1} we see that $g_m(n)=n$ and $r$ is the smallest integer such that $n^r$ is a $m$-th power. Note that by definition, $r=T_m(n)<m$. By Lemma \ref{lemma:m-th power} this means that $n=k^d$ for some $k,d\in\mathbb{N}$ such that $\gcd(m,d)>1$. It follows that $k^{dr}$ is a $m$-th power, so $m|dr$. 
    
    Assume $r\nmid m$. Then there exists a prime $p$ such that $m|d\frac{r}{p}$. However then $k^{d\frac{r}{p}}=n^{\frac{r}{p}}$ is an $m$-th power and so $T_m(n)\leq \frac{r}{p}<r$, which is a contradiction. Therefore $r|m$.
    
    Conversely if $r|m$ and $r<m$, then there are infinitely many $n$ such that $n=\prod_{i=1}^s p_i^{e_i}$ and $\gcd(e_1,\ldots,e_s,m)=\frac{m}{r}>1$. (Eg. $e_1=\ldots=e_s=t\frac{m}{r}$ for any $t$ such that $\gcd(t,m)=1$). By Lemma \ref{lemma:prime2}, it follows that these are exactly the $n$ such that $T'_m(n)=1$ and $T_m(n)=\frac{m}{\gcd(e_1,\ldots,e_s,m)}=r$. 
\end{proof}

On the OEIS page for $g_2(n)$ \cite{citation-key}, Robert G. Wilson v. observed that there did not seem to be any $n$ such that $T_2(n)=2$. It was proved in \cite{kagey2024conjecture} that there does not exist any such $n$ and the authors conjectured that for all $k\neq 2$, there exists $n$ such that $T_2(n)=k$. They also asked if there were any notable properties of $T_m(n)$ or $T'_m(n)$ for other $m$.

We show that for other $m$ we can have $T'_m(n)=2$ and $T_m(n)=2$. For instance, it is clear that for even $m\geq 4$ and $x\in\mathbb{N}$, we have $T_{m}(x^{\frac{m}{2}})=2$. However it is more interesting to look at cases where $n$ is not a power and $m$ is prime.

\begin{lemma}\label{lemma:18} We have that:
\begin{enumerate}
        \item  $T'_{5}(18)=2$

        \item  $T'_{11}(48)=2$

        \item  $T_{3}(2)=T'_{3}(2)=2$

        \item $T_{5}(4)=T'_{5}(4)=2$

        \item $T_{3}(12)=T'_{3}(12)=2$
    \end{enumerate}
\end{lemma}
\begin{proof}
   In the Appendix, we show that $g_5(18)=24$ and $18^{\{2\}},24$ is a corresponding sequence for $g_5(18)$. By Lemma \ref{lemma:prime1}, $T'_5(18)\neq 1$ as $g_5(18)\neq 18$. So $T'_5(18)=2$.

    In the Appendix, we show that $g_{11}(48)=54$ and $48^{\{2\}},54^{\{3\}}$ is a corresponding sequence for $g_{11}(48)$. By Lemma \ref{lemma:prime1}, $T'_{11}(48)\neq 1$ as $g_{11}(48)\neq 48$. So $T'_{11}(48)=2$.

    In the Appendix, we show that $g_3(2)=4$ and $T'_3(2)\leq T_3(2)\leq 2$. As $g_3(2)\neq 2$, it follows from Lemma \ref{lemma:prime1} that $T_3(2)\geq T'_3(2)>1$. Therefore $T_3(2)=T'_3(2)=2$.

   In the Appendix, we show that $g_5(4)=8$ and $T'_5(4)\leq T_5(4)\leq 2$. As $g_5(4)\neq 4$, it follows from Lemma \ref{lemma:prime1} that $T_5(4)\geq T'_5(4)>1$. Therefore $T_5(4)=T'_5(4)=2$.

    In the Appendix, we show that $g_3(12)=18$ and $T'_3(12)\leq T_3(12)\leq 2$. As $g_3(12)\neq 12$, it follows from Lemma \ref{lemma:prime1} that $T_3(12)\geq T'_3(12)>1$. Therefore $T_3(12)=T'_3(12)=2$. .
\end{proof}

\begin{lemma}\label{lemma:k}
    Let $k\geq 1$. If $g_m(n)=g_{km}(n)\neq n$ and $T'_m(n)=2$ then $T'_{km}(n)=2$.
\end{lemma}
\begin{proof}
    By Lemma \ref{lemma:T bound}, we have $T'_{km}(n)\leq T'_m(n)=2$. As $g_{km}(n)\neq n$, we see by Lemma \ref{lemma:prime1} that $T'_{km}(n)>1$. Therefore $T'_{km}(n)=2$.
\end{proof}

\begin{corollary}
    We see that for any $k\geq 1$:

    \begin{enumerate}
        \item  $T'_{5k}(18)=2$

        \item  $T'_{11k}(48)=2$

        \item  $T'_{3k}(2)=2$

        \item $T'_{5(2k-1)}(4)=2$

        \item $T'_{3k}(12)=2$.
    \end{enumerate}
\end{corollary}
\begin{proof}
    In the Appendix, we show that $g_5(18)=g_{5k}(18)\neq 18$. Then by Lemmas \ref{lemma:18} and \ref{lemma:k}, we see that $T'_{5k}(18)=2$.

    In the Appendix, we show that $g_{11}(48)=g_{11k}(48)\neq 48$. Then by Lemmas \ref{lemma:18} and \ref{lemma:k}, we see that $T'_{11k}(48)=2$.

    In the Appendix, we show that $g_3(2)=g_{3k}(2)\neq 2$. Then by Lemmas \ref{lemma:18} and \ref{lemma:k}, we see that $T'_{3k}(2)=2$.

    In the Appendix, we show that $g_5(4)=g_{5(2k-1)}(4)\neq 4$. Then by Lemmas \ref{lemma:18} and \ref{lemma:k}, we see that $T'_{5(2k-1)}(4)=2$.

    In the Appendix, we show that $g_3(12)=g_{3k}(12)\neq 12$. Then by Lemmas \ref{lemma:18} and \ref{lemma:k}, we see that $T'_{3k}(12)=2$.
\end{proof}

\begin{conjecture}
    We conjecture that for each $m\geq 2$, there are only finitely many $n$ such that $T_m(n)=T'_m(n)=2$.
\end{conjecture}

We would expect it to be rare just for $T'_m(n)=2$ when $n$ is large, as it requires $n$ to be closely followed by an integer with the same set of prime divisors that are raised to powers not divisible by $m$. 

By Lemmas \ref{lemma:clear} and \ref{lemma:prime3}, our conjecture asserts that for each $m$, there are finitely many $n$ that satisfy $T_m(n)=2$ which are not $\frac{m}{2}$-th powers. 

It would be of interest to know, for fixed $m$, which values of $(T_m(n),T'_m(n))=(a,b)$ occur and if they occur for infinitely many $n$. 

\bigskip

\textbf{Acknowledgements:} We thank Peter Kagey for his encouragement and suggestions in writing this paper.

\bibliographystyle{plain}
\bibliography{Latex}

\begin{thebibliography}{1}

\bibitem{citation-key}
OEIS Foundation Inc. (2024), The On-Line Encyclopedia of Integer Sequences.

\bibitem{MR4706772}
Hung~M. Bui, Kyle Pratt, and Alexandru Zaharescu.
\newblock A problem of {E}rdős-{G}raham-{G}ranville-{S}elfridge on integral points on hyperelliptic curves.
\newblock {\em Math. Proc. Cambridge Philos. Soc.}, 176(2):309--323, 2024.

\bibitem{dubner1996large}
Harvey Dubner.
\newblock Large {S}ophie {G}ermain primes.
\newblock {\em Mathematics of computation}, 65(213):393--396, 1996.

\bibitem{Erdos}
Paul Erdős and J.~L. Selfridge.
\newblock 6655.
\newblock {\em The American Mathematical Monthly}, 99(8):791--794, 1992.

\bibitem{doi:10.1080/0025570X.1986.11977243}
Ronald Graham.
\newblock Problems.
\newblock {\em Mathematics Magazine}, 59(3):172, 1986.

\bibitem{hardy1923some}
Godfrey~H Hardy and John~E Littlewood.
\newblock Some problems of ‘{P}artitio numerorum’; {III}: {O}n the expression of a number as a sum of primes.
\newblock {\em Acta Mathematica}, 44(1):1--70, 1923.

\bibitem{kagey2024conjecture}
Peter Kagey and Krishna Rajesh.
\newblock On a {C}onjecture about {R}on {G}raham's {S}equence.
\newblock {\em arXiv preprint arXiv:2410.04728}, 2024.

\bibitem{efdab2f9-940e-3325-88c5-7e0949aa0d3d}
Michael Reid.
\newblock Bijection between {I}ntegers and {C}omposites.
\newblock {\em Mathematics Magazine}, 60(3):180, 1987.

\end{thebibliography}

\section{Appendix}

Let $a_1^{\{r_1\}},\ldots, a_t^{\{r_t\}}$ be a $m$-product sequence. Then $\prod_{i=1}^t a_i^{r_i}$ is an $m$-th power. Equivalently, for any prime $p$ such that $p|a_i$ for some $1\leq i\leq t$, we have $p^{km}||\prod_{i=1}^t a_i^{r_i}$ for some $k\in\mathbb{N}$. Let $q_1,\ldots,q_s$ be the set of primes that divide $\prod_{i=1}^t a_i^{r_i}$ and so let $\prod_{i=1}^t a_i^{r_i}=\prod_{j=1}^s q_j^{e_j}$. Then $e_j\equiv 0\mod m$ for all $1\leq j\leq s$. Let $A_j=\{a_i\mid 1\leq i\leq t, q_j|a_i\}$ and if $q_j^l||a_i$, define $v_{q_j}(a_i)=l$. Then let $v(A_j)=\sum_{a_i\in A_j}v_{q_j}(a_i)r_i$. Then we see that $v(A_j)=e_j\equiv 0\mod m$. It is clear that $\prod_{i=1}^t a_i^{r_i}$ is an $m$-th power if and only if $v(A_j)\equiv 0\mod m$ for all $1\leq j\leq s$.

So fix $m$ and consider $a_1=n, a_2=n+1,\ldots, a_t=k$ to be the integers in $[n,k]$. Then letting each $0\leq r_i\leq m-1$ be a variable, we can see whether there exists a solution $(r_1,\ldots,r_t)$. A solution exists with $1\leq r_1\leq m-1$ if and only if $g_m(n)\leq k$. We must have $r_1\geq 1$ as $n$ must be present in the sequence. 

Let $n=a_1^{\{r_1\}},\ldots,a_t^{\{r_t\}}=g_m(n)$ be a $m$-product sequence for $g_m(n)$ where $1\leq r_1\leq m-1$ and $0\leq r_i\leq m-1$ for $2\leq i\leq t$. Note that $m\nmid r_1$. If we say $r_i=m-20$ and $m<20$ then we mean $r_i\equiv m-20\mod m$ and $r_i\in\{0,\ldots,m-1\}$. Similarly, if we say $r_i=a$ for some $a\in\mathbb{N}$, then we mean $r_i\equiv a\mod m$ and $r_i\in\{0,\ldots,m-1\}$.

We include the results for $0\leq n\leq 7$ and $n=12,18,48$.

\begin{enumerate}
\setcounter{enumi}{-1}
\item $g_m(0)$: By Lemma \ref{lemma:m-th power}, $g_m(0)=0$ for all $m\geq 2$.

\item $g_m(1)$: By Lemma \ref{lemma:m-th power}, $g_m(1)=1$ for all $m\geq 2$.

\item $g_m(2)$: As $2$ is prime, by Lemma \ref{lemma:2p} it follows that $g_m(2)\geq 4$.

If $m=2$, by Lemma \ref{lemma:not m-th power} we do not need to include $4$, so $g_2(2)\geq 6$. As $2\cdot 3\cdot 6=6^2$, we see that $g_2(2)=6$.

If $m>2$ is odd, then $2\cdot 4^{\frac{m-1}{2}}=2^m$. If $m>2$ is even then $2^2\cdot 4^{\frac{m-2}{2}}=2^m$. So $g_m(2)=4$ for $m>2$.

\item $g_m(3)$: As $3$ is prime, by Lemma \ref{lemma:2p} we see that $g_m(3)\geq 6$. If $m>2$, then $3^2\cdot 4 \cdot 6^{m-2}=6^m$, so $g_m(3)=6$.

Let $m=2$. If $6$ is in the sequence then as $2||6$, there must be a distinct even term in the sequence. By Lemma \ref{lemma:not m-th power}, we ignore $4$ so $a_t\geq 8$. If $6$ is not in the sequence, then $a_t\geq 9$ as a multiple of $3$ other than $3$ is in the sequence. Therefore $g_2(3)\geq 8$. We see that $3\cdot 6\cdot 8=12^2$ so $g_2(3)=8$.

\item $g_m(4)$: 
By Lemma \ref{lemma:m-th power}, we have $g_m(4)=4$ if $2|m$. 

If $2\nmid m$, then by Lemma \ref{lemma:p} there must be a distinct even term in any $m$-product sequence starting with $4$. If $6$ is in the sequence, then by Lemma \ref{lemma:p}, we see that $a_t\geq 9$. If $6$ is not in the sequence then another even term must be included, so $a_t\geq 8$. Therefore $g_m(4)\geq 8$.

If $m>3$ and $2\nmid m$, then as $4^{\frac{m-3}{2}}\cdot 8=2^m$, we have $g_m(4)=8$.

If $m=3$ then by Lemma \ref{lemma:not m-th power}, we ignore $8$, so $g_3(4)\geq 9$. Then as $4\cdot 6\cdot 9=6^3$, we see that $g_3(4)=9$.

\item $g_m(5)$: By Lemma \ref{lemma:2p}, it follows that $g_m(5)\geq 10$. If $m$ is odd then $5\cdot 6\cdot 9^{\frac{m-1}{2}} \cdot 10^{m-1}=30^m$ and if $m$ is even then $5^{\frac{m}{2}}\cdot 8^\frac{m}{2}\cdot 10^\frac{m}{2}=20^m$. So $g_m(5)=10$.

\item $g_m(6)$: By Lemma \ref{lemma:p}, we see that $g_m(6)\geq 9$. 

As $6^6\cdot 8^{m-2}\cdot 9^{m-3}$ is an $m$-th power, we see that $g_m(6)=9$ if $m\neq 2,3,6$.

Let $m=2$. Then by Lemma \ref{lemma:not m-th power}, we ignore $9$. As we must include another multiple of $3$, we have $g_2(6)\geq 12$. As $6\cdot 8\cdot 12=24^2$, we see that $g_2(6)=12$.

Let $m=3$. Then by Lemma \ref{lemma:not m-th power}, we ignore $8$. By Lemma \ref{lemma:p}, we must include an even term greater than $6$. If we include $10$, then by \ref{lemma:p} we must have $a_t\geq 15$. If we do not include $10$ then $a_t\geq 12$ because we must include an even term distinct from $6$. As $6^2\cdot 9\cdot 12^2=36^3$, we see that $g_3(6)=12$. 

Let $m=6$. Assume that $g_6(6)\leq 11$. If we use $7, 10$ or $11$ in the $6$-product sequence then by Lemma \ref{lemma:p} we would have $a_t\geq 14$. So consider $a_1=6$, $a_2=8$ and $a_3=9$. Then $v(A_2)=r_1+3r_2$ and $v(A_3)=r_1+2r_3$. So if $v(A_2)\equiv 0\mod m$ then $3|r_1$ and if  $v(A_3)\equiv 0\mod m$ then $2|r_1$. Therefore if $6^{\{r_1\}},8^{\{r_2\}},9^{\{r_3\}}$ is a $6$-product sequence then $r_1\equiv 0 \mod 6$. But as $1\leq r_1\leq 5$, this is not possible. So $g_6(6)\geq 12$. As $6^2\cdot 9\cdot 12^2=6^6$, we have $g_6(6)=12$.

\item $g_m(7)$: By Lemma \ref{lemma:2p}, it follows that $g_m(7)\geq 14$. As $7^3\cdot 8\cdot 14^{m-3}=14^m$ and $7^2\cdot9\cdot 12\cdot 14=42^3$, we see that $g_m(7)=14$. 

\item $g_m(12)$: By Lemma \ref{lemma:p} we must have a distinct integer divisible by $3$ in the sequence, so $a_t\geq 15$. If we include $15$ then by Lemma \ref{lemma:p} we similarly must have $a_t\geq 20$. If we do not include $15$ then as we must include a multiple of $3$, we have $a_t\geq 18$. So $g_m(12)\geq 18$. When $m\neq 2,4,8$, we see that $12^{8}\cdot 16^{m-3}\cdot 18^{m-4}$ is a $m$-th power so $g_m(12)=18$.

Let $m=8$. If we include $13,14,15$ or $17$ we see by Lemma $\ref{lemma:p}$ that $g_8(12)\geq 20$. We have shown that $a_t\geq 18$ and if $g_8(12)=18$ then the sequence would have $a_1=12, a_2=16$ and $a_3=18$. Then $v(A_2)=2r_1+4r_2+r_3$ and $v(A_3)=r_1+2r_3$. If $v(A_2)\equiv 0\mod 8$ then $2|r_3$ and so if $v(A_3)\equiv 0\mod 8$ then $4|r_1$. Then $0\equiv v(A_2)\equiv 2r_1+4r_2+r_3\equiv 4r_2+r_3\mod 8$ so $4|r_3$ and it follows that $0\equiv v(A_3)\equiv r_1+2r_3\equiv r_1\mod 8$. So $r_1\equiv 0\mod 8$ which is a contradiction. Then $g_8(12)>18$. If $19$ is in the sequence then by Lemma \ref{lemma:p} we see that $a_t\geq 38$. So $a_t\geq 20$ and as $12\cdot15^7\cdot16\cdot20=30^8$, we have $g_8(12)=20$.

Let $m=4$. Then by Lemma \ref{lemma:divisible} we see that $g_4(12)\geq g_8(12)=20$ and as $12\cdot15^7\cdot16\cdot20=900^4$, we have $g_4(12)=20$.

Let $m=2$. Then by Lemma \ref{lemma:divisible} we see that $g_2(12)\geq g_8(12)=20$ and as $12\cdot15\cdot20=60^2$, we have $g_2(12)=20$.

\item $g_m(18)$: By Lemma \ref{lemma:p}, then the sequence must include a multiple of $2$ distinct from $18$. If $20$ appears in the sequence, then there must be another multiple of $5$ in the sequence, so $a_t\geq 25$. If $22$ appears in the sequence, then there must be another multiple of $11$ in the sequence, so $a_t\geq 33$. Therefore $a_t\geq 24$ and as $18^{\frac{2m}{5}}\cdot 24^{\frac{m}{5}}=6^m$ when $5|m$, we have $g_m(18)=24$ in this case.

Let $5\nmid m$. If $19,21,22$ or $23$ were in the sequence then $a_t\geq 28$. If $20$ was in the sequence then $a_t\geq 25$. So if $g_m(18)=24$ then $a_1=18$ and $a_2=24$. Then $v(A_2)=r_1+3r_2$ and $v(A_3)=2r_1+r_2$. If $v(A_2)\equiv 0\mod m$ then $0\equiv 2v(A_2)\equiv 2r_1+6r_2\mod m$. Then if $v(A_3)\equiv 0\mod m$ then $0\equiv 2v(A_2)-v(A_3)\equiv 5r_2\mod m$. Then as $5\nmid m$, we have $r_2\equiv 0\mod m$ and so $r_1\equiv 0\mod m$. This is a contradiction so $g_m(18)\geq 25$. As $18^4\cdot 20^{10}\cdot 24^{m-8}\cdot 25^{m-5}$ is a $m$-th power, we see that $g_m(18)=25$ when $m\neq 2,4$ and $5\nmid m$. 

Let $m=4$. If $g_4(18)=25$ then the sequence would have to be $a_1=18$, $a_2=20$, $a_3=24$ and $a_4=25$. So $v(A_2)=r_1+2r_2+3r_3$, $v(A_3)=2r_1+r_3$ and $v(A_5)=r_2+2r_4$. If $v(A_5)\equiv 0\mod 4$ then $2|r_2$. Then if $v(A_2)\equiv 0\mod 4$ it follows that $0\equiv r_1+2r_2+3r_3\mod 4\equiv r_1+3r_3\mod 4$. If $v(A_3)\equiv 0\mod 4$ then $2|r_3$ and as $r_1+3r_3\equiv 0\mod 4$, we see that $2|r_1$. Then $0\equiv 2r_1+r_3\equiv r_3\mod 4$ and so $0\equiv r_1+3r_3\equiv r_1\mod 4$. This is a contradiction, so $g_4(18)\geq 26$. If $26$ is in the sequence then $a_t\geq 39$. Therefore $g_4(18)\geq 27$ and as $18^2\cdot 24^2\cdot 27^2=108^4$, it follows that $g_4(18)=27$. By Lemma \ref{lemma:divisible} we see that $g_2(18)\geq g_4(18)=27$ and as $18\cdot 24\cdot 27=108^2$, it follows that $g_2(18)=27$.

\item $g_m(48)$: By Lemma \ref{lemma:p}, we see that there must be a multiple of $3$ greater than $48$ in the sequence. If $51=3\cdot17$ is in the sequence, then $a_t\geq 68=4\cdot17$. So $g_m(48)\geq 54$ and as $48^\frac{m}{11}\cdot54^\frac{7m}{11}$ is a $m$-th power, we see that $g_m(48)=54$ if $11|m$. As $48^\frac{m}{2}\cdot50^\frac{m}{2}\cdot54^\frac{m}{2}$ is a $m$-th power, we see that $g_m(48)=54$ if $2|m$. 

Let $11\nmid m$ and $2\nmid m$. Assume $g_m(48)=54$. If $51,52$ or $53$ are in the sequence then $a_t\geq 65$. So the sequence is $a_1=48$, $a_2=49$, $a_3=50$ and $a_4=54$. Then $v(A_2)=4r_1+r_3+r_4$, $v(A_3)=r_1+3r_4$, $v(A_5)=2r_3$ and $v(A_7)=2r_2$. If $v(A_5)\equiv 0\mod m$, then as $2\nmid m$ we see that $r_3\equiv 0\mod m$. If $v(A_2)\equiv 0\mod m$, then $0\equiv 3v(A_2)\equiv 12r_1+3r_4\mod m$. If $v(A_3)\equiv 0\mod m$, then $0\equiv 12r_1+3r_4-v(A_3)\equiv 11r_1\mod m$. As $11\nmid m$, this means $r_1\equiv 0\mod m$ which is a contradiction. Therefore $g_m(48)\geq 55$. If $55$ is in the sequence, then $a_t\geq 66$. Therefore $g_m(48)\geq 56$ and as $48^{18}\cdot49^{11}\cdot54^{m-6}\cdot56^{m-22}$, we see that $g_m(48)=56$ if $11\nmid m$, $2\nmid m$ and $m\neq 3,9$.

Let $m=3$. Assume $g_3(48)=56$. If 50 is in the sequence, then we need to include another multiple of 5 by Lemma \ref{lemma:p}, so either 55 is in the sequence or $a_t\geq 60$. If 55 is in the sequence, then $a_t\geq 66$. If 51,52,53 or 55 are in the sequence, then $a_t\geq 65$. Therefore the sequence is $a_1=48$, $a_2=49$, $a_3=54$ and $a_4=56$. So $v(A_2)=4r_1+r_3+3r_4$, $v(A_3)=r_1+3r_3$ and $v(A_7)=2r_2+r_4$. If $v(A_3)\equiv 0\mod 3$, then $0\equiv r_1+3r_3\equiv r_1\mod 3$. This is a contradiction, so $g_3(48)\geq 57$. If $57,58$ or $59$ are in the sequence, then $a_t\geq 76$. So $g_3(48)\geq 60$, and as $48\cdot49\cdot50^2\cdot54^2\cdot56\cdot60^2$ is a cube, we see that $g_3(48)=60$. 

Let $m=9$. Assume $g_9(48)=56$. By the above, the sequence is $a_1=48$, $a_2=49$, $a_3=54$ and $a_4=56$. So $v(A_2)=4r_1+r_3+3r_4$, $v(A_3)=r_1+3r_3$ and $v(A_7)=2r_2+r_4$. If $v(A_3)\equiv 0\mod 9$, then $3|r_1$ and so if $v(A_2)\equiv 0\mod 9$, we see that $3|r_3$. Then as $v(A_3)\equiv r_1+3r_3\equiv 0\mod 9$, it follows that $r_1\equiv 0\mod 9$. This is a contradiction, so $g_9(48)\geq 57$. If $57,58$ or $59$ are in the sequence, then $a_t\geq 76$. So $g_9(48)\geq 60$ and by Lemma \ref{lemma:divisible}, we see that $g_9(48)\leq g_3(48)=60$. Therefore $g_9(48)=60$. 
\end{enumerate}

\end{document}